\documentclass[a4paper,12pt]{amsart}

\usepackage{amsfonts}
\usepackage{amsmath}
\usepackage{amssymb}
\usepackage{mathrsfs}
\usepackage[colorlinks]{hyperref}

\usepackage{graphicx}

\usepackage[usenames]{color}

\newtheorem{thm}{Theorem}[section]

\newtheorem{Assum}[thm]{Assumption}

\numberwithin{equation}{section}

\theoremstyle{definition}

\begin{document}

\title[Direct and inverse source problems of  $q$-heat  equation ]
{Direct and inverse source problems for heat  equation in quantum calculus}

\author[Michael Ruzhansky]{Michael Ruzhansky}
\address{
 Michael Ruzhansky:
  \endgraf
Department of Mathematics: Analysis, Logic and Discrete Mathematics,
  \endgraf
 Ghent University, Ghent,
 \endgraf
  Belgium 
  \endgraf
  and 
 \endgraf
 School of Mathematical Sciences, Queen Mary University of London, London,
 \endgraf
 UK  
 \endgraf
  {\it E-mail address} {\rm michael.ruzhansky@ugent.be}
  }

\author[Serikbol Shaimardan]{Serikbol Shaimardan}
\address{
  Serikbol Shaimardan:
  \endgraf
  L. N. Gumilyov Eurasian National University, Astana,
  \endgraf
  Kazakhstan 
  \endgraf
  and 
  \endgraf
Department of Mathematics: Analysis, Logic and Discrete Mathematics
  \endgraf
 Ghent University, Ghent,
 \endgraf
  Belgium
  \endgraf
  {\it E-mail address} {\rm shaimardan.serik@gmail.com} 
  }

\thanks{The authors are supported by the FWO Odysseus 1 grant G.0H94.18N: Analysis and Partial Differential Equations and by the Methusalem programme of the Ghent University Special Research Fund (BOF) (Grant number 01M01021). Michael Ruzhansky is also supported by EPSRC grant EP/R003025/2, and  the second author  by the international internship program “Bolashak” of the Republic of Kazakhstan.}

\date{}

\begin{abstract}
In this paper we explore the  weak solutions of the Cauchy problem and an inverse source problem for the heat equation in the quantum calculus, formulated in abstract Hilbert spaces.  For this we use the Fourier series expansions.  Moreover,  we prove  the existence, uniqueness and stability of the weak solution of the inverse problem with   a final determination condition. We give some examples such as  the $q$-Sturm–Liouville problem, the $q$-Bessel operator,  the $q$-deformed Hamiltonian, the fractional Sturm-Liouville operator, and  the restricted fractional Laplacian, covered by our analysis.

\end{abstract}

\subjclass[2010]{34C10, 39A10, 26D15.}

\keywords{Rubin difference operator, heat equation, wave equation, a priori estimate, $q$-derivative, $q$-calculus, well-posedness, Sobolev type space}

\maketitle

\section{Introduction}

It is well known that quantum groups provide the key to $q$-deforming the fundamental structures of physics from the point of view of the non-commutative geometry. An important concept in this theory is integration and Fourier theory on quantum spaces, see e.g. \cite{St1996}, \cite{Wa2003} and \cite{Wa2004}. Nowadays, the theory of quantum groups and $q$-deformed algebras have been the subject of intense investigation. Many physical applications have been investigated on the basis of the $q$-deformation of the Heisenberg algebra \cite{BM2000} and \cite{HW1999}. For instance, the $q$-deformed Schr$\ddot{o}$dinger equations have been proposed in  \cite{Micu1999} and \cite{Lavagno2009}  and applications to the study of the $q$-deformed version of the hydrogen atom and of the quantum harmonic oscillator \cite{EM2006} have been introduced. The fractional calculus and $q$-deformed Lie algebras are closely related.  A new class of fractional $q$-deformed Lie algebras is considered, which for the first time allows a smooth transition between different Lie algebras \cite{Herrmann2010}.

The origin of the  $q$-difference calculus can be traced back to the works   \cite{J1908, J1910} by  F.~Jackson and  R. D.~Carmichael \cite{C1912} from the beginning of the twentieth century, while basic definitions and properties can be found e.g. in the monographs \cite{CK2000, E2002}.  Recently,  the fractional $q$-difference calculus has been proposed by W. Al-salam  \cite{A1966} and R. P.~Agarwal \cite{A1969}. Today, maybe due to the explosion in research within the fractional differential calculus setting, new developments in the theory of the fractional $q$-difference calculus have been addressed extensively by several researchers. For example, some researchers obtained $q$-analogues of the integral and differential fractional operators  such as the $q$-Laplace transform and $q$-Taylor’s formula \cite{PMS2007}, $q$-Mittag-Leffler function \cite{A1966}. Moreover, in 2007, M. S.~Ben Hammouda and Akram Nemri defined the higher-order $q$-Bessel translation and the higher order $q$-Bessel Fourier transform and established some of their properties, as well as studied  the higher-order $q$-Bessel heat equation \cite{HM2007}. In 2012, A. Fitouhi and F. Bouzeffour established in great detail the $q$-Fourier analysis related to the $q$-cosine and constructed  the $q$-solution source, the $q$-heat polynomials, and solved the $q$-analytic Cauchy problem.

The paper is organized as follows: The main results are presented and proved in Sections \ref{S3} and \ref{S4}  for direct and inverse problems, respectively. In the final section the  examples are given. In order to simplify these presentations we include in Section \ref{S2} the 
 necessary preliminaries.

\section{Preliminaries}\label{S2}

In this section, we recall some  notations  related to the  $q$-calculus. We will always assume that $0<q<1$. The $q$-real number $[\alpha ]_q$ is defined by
$$
[\alpha ]_{q}:=\frac{1-q^{\alpha }}{1-q}.
$$

The $q$-analogue differential operator $D_{q}f(x)$ is defined by 
\begin{eqnarray}\label{additive2.1}
D_{q}f(x)=\frac{f(x)-f(qx)}{x(1-q)},
\end{eqnarray}

The $q$-derivative of a product of two functions has the form
\begin{eqnarray}\label{additive2.2}
D_{q}(fg)(x)= f(qx)D_{q}(g)(x)+D_{q}(f)(x)g(x).
\end{eqnarray}

The $q$-integral (or Jackson integral) is defined by (see \cite{J1910})
\begin{eqnarray}\label{additive2.4}
\int\limits_0^x f(t)d_{q}t=(1-q)x\sum\limits_{m=0}^\infty q^{m}f(xq^{m}),
\end{eqnarray}
and, more generally, 
\begin{eqnarray*} 
\int\limits_a^b f(x)d_{q}x=\int\limits_0^b f(x)d_{q}x-
\int\limits_0^a f(x)d_{q}x,
\end{eqnarray*}
provided the sums converge absolutely. Note that 
\begin{eqnarray}\label{additive2.4}
\int\limits_a^x D_qf(t)d_{q}t= f(x) -f(a). 
\end{eqnarray}

The  $q$-version of the integration by parts takes the form 
\begin{eqnarray}\label{additive2.5}
\int\limits_a^bf(x)D_qg(x)d_qx=\left[fg\right]_a^b-\int\limits_a^bg(qx)D_qf(x)d_qx.   
\end{eqnarray}

Let $H$ be a separable Hilbert space and let $\mathcal{L}$ be an operator with the  discrete spectrum and eigenfunctions  $\{\phi_k\}_{k\in{I}}$ in  $H$,   where $I$ is a countable set $(I=\mathbb{N}^k$ or $I=\mathbb{Z}^k$ for some $k$). We assume that the operator  $\mathcal{L}$ is  diagonalisable (can be written in the infinite dimensional matrix form) with respect to the  orthonormal basis $\{\phi_k\}_{k\in{I}}$ of  $H$ with the eigenvalues $\lambda_k$ such that $\lambda_k>\lambda_0>0$, where $\lambda_0$  is  some constant  independent of  $k$. In further calculus for our analysis we will also require that $\lambda_k\rightarrow\infty$. The Plancherel identity takes the form
\begin{eqnarray}\label{additive2.6}
 \|u\|_{H}=\left(\sum\limits_{k\in{I}} \left|\langle u,\phi_k\rangle_H\right|^2 \right)^\frac{1}{2},
 \end{eqnarray}
for $u\in{H}$.

We denote  the domain of the operator $\mathcal{L}^k$ for $k\in\mathbb{N}_0$ by 
\begin{eqnarray*}
Dom(\mathcal{L}^k):=\left\{u\in{H}: \mathcal{L}^iu\in Dom(\mathcal{L}), \;\;\;i=0,1,2,\cdots,k-1\right\}.
\end{eqnarray*}
The space $C^\infty_{\mathcal{L}}(H):= \bigcap\limits_{k=1}^\infty{Dom}(\mathcal{L}^k)$ is called the space of test functions for $\mathcal{L}$.  Moreover, we introduce the Fr$\acute{e}$chet topology of $C^\infty_{\mathcal{L}}(H)$ with the family of norms
\begin{eqnarray*}
\|u\|_{C^k_{\mathcal{L}}}:=\max\limits_{j\leq{k}}\|\mathcal{L}^ju\|_H,
\end{eqnarray*}
for $k\in\mathbb{N}_0$, $u\in C^\infty_{\mathcal{L}}(H)$, and the space of $\mathcal{L}$-distributions $\mathcal{D}'_{\mathcal{L}}(H):=\mathfrak{L}\left(C^\infty_{\mathcal{L}}(H),\mathbb{C}\right)$. Consequently, we can also define Sobolev spaces $\mathcal{H}^d_{\mathcal{L}}$   associated to $\mathcal{L}$ as   
\begin{eqnarray*}
\mathcal{H}^d_{\mathcal{L}}:=\left\{u\in\mathcal{D}'_{\mathcal{L}}(H): \mathcal{L}^{d/2}u\in{H}\right\}
\end{eqnarray*}
for any $d\in\mathbb{R}$. Using Plancherel’s identity (\ref{additive2.6}), we can write the norm in the following form: 
\begin{eqnarray}\label{additive2.6A}
\|u\|_{\mathcal{H}^d_{\mathcal{L}}}:=\|\mathcal{L}^{d/2}u\|_H=\left(\sum\limits_{k\in{I}}\lambda_k^d\left|\langle u,\phi_k\rangle_H\right|^2\right)^\frac{1}{2}.
\end{eqnarray}

We introduce the spaces   $C_q^k\left( [0, T]; \mathcal{H}^d_{\mathcal{L}}\right)$   and $L^\infty\left( [0, T]; \mathcal{H}^d_{\mathcal{L}}\right)$ as 
\begin{eqnarray*}
\|u\|_{C_q^k\left( [0, T];\mathcal{H}^d_{\mathcal{L}}\right)}:=
\sum\limits_{m=0}^k\max\limits_{0\leq t\leq T}\| D_q^mu(t)\|_{\mathcal{H}^d_{\mathcal{L}}},
\end{eqnarray*}
and 
\begin{eqnarray*}
\|u\|_{L^\infty\left( [0, T];\mathcal{H}^d_{\mathcal{L}}\right)}:=
\sup\limits_{0\leq t\leq T}\|  u(t)\|_{\mathcal{H}^d_{\mathcal{L}}},
\end{eqnarray*}
respectively, where  $D_q^ku=D_q\left(D_q^{k-1}u\right)$.

As given in \cite{CK2000}, two  $q$-analogues of the exponential function are defined as  
 \begin{eqnarray}\label{additive2.6A}
 e_q^x=\sum\limits_{k=0}^\infty\frac{x^k}{[k]_q}, \;\;\;\;
  E_q^x=\sum\limits_{k=0}^\infty q^{k(k-1)/2}\frac{x^k}{[k]_q}.
  \end{eqnarray}

Moreover, we can rewrite   
\begin{eqnarray*}
E_q^x&=&\lim\limits_{N\rightarrow\infty}\prod\limits_{i=0}^N\left(1+(1-q)q^ix\right),
\end{eqnarray*}   
and we have 
\begin{eqnarray*}
 e_q^xE_q^{-x}=1.
  \end{eqnarray*}

We will be making the following assumption:
\begin{Assum}\label{a2.1}  We assume that  $\upsilon$ is a  continuous function on $[0,T]$ and, there are some $ \beta, \alpha>0$ such that  $0<\alpha\leq\upsilon(t)\leq\beta$  for all  $t\in[0, T]$. 
\end{Assum}

For $0\leq{t}\leq{T}$ we denote  
\begin{eqnarray}\label{additive2.7}
\gamma_\upsilon(t):=\frac{1}{ \lim\limits_{N\rightarrow\infty}\prod\limits_{i=0}^N\left(1+(1-q) tq^i\upsilon(q^it)\right)}.
\end{eqnarray}

So, by Assumption \ref{a2.1}, we obtain
\begin{eqnarray*} 
\frac{1}{ \prod\limits_{i=0}^N\left(1+(1-q) tq^i\beta\right)} \leq\frac{1}{ \prod\limits_{i=0}^N\left(1+(1-q) tq^i\upsilon(q^it)\right)}\leq\frac{1}{ \prod\limits_{i=0}^N\left(1+(1-q) tq^i\alpha\right)}, 
\end{eqnarray*}
and  it can be written as 
\begin{eqnarray}\label{additive2.7A}
\frac{1}{E_q^{\beta{t}}}\leq\gamma_\upsilon(t)\leq\frac{1}{E_q^{\alpha{t}}}.
\end{eqnarray}

Form  (\ref{additive2.1}) and  (\ref{additive2.4}) it follows that  
\begin{eqnarray}\label{additive2.8}
D_q\left[\gamma^{-1}_\upsilon(\lambda_kt)\right]=\lambda_k\upsilon(t)\gamma^{-1}_\upsilon(q\lambda_kt)
\end{eqnarray}
and 
\begin{eqnarray}\label{additive2.9}
\int\limits_0^T \gamma^{-1}_\upsilon(\lambda_ks) d_qs&=&\frac{1}{\lambda_k}\int\limits_0^T \frac{D_q\left[\gamma^{-1}_\upsilon(\lambda_ks)\right]}{\upsilon(t)}d_qs\nonumber\\
&\geq&\frac{1}{\lambda_k\beta}\int\limits_0^T  D_q\left[\gamma^{-1}_\upsilon(\lambda_ks)\right] d_qs\nonumber\\
&=&\frac{\gamma^{-1}_\upsilon(\lambda_kT)}{\lambda_k\beta}
\end{eqnarray}
and 
\begin{eqnarray}\label{additive2.10}
\int\limits_0^t \gamma^{-1}_\upsilon(\lambda_ks) d_qs&=&\frac{1}{\lambda_k}\int\limits_0^t \frac{D_q\left[\gamma^{-1}_\upsilon(\lambda_ks)\right]}{\nu(t)}d_qs\nonumber\\
&\leq&\frac{1}{\lambda_k\alpha}\int\limits_0^T  D_q\left[\gamma^{-1}_\upsilon(\lambda_ks)\right] d_qs\nonumber\\
&=&\frac{\gamma^{-1}_\upsilon(\lambda_kt)}{\lambda_k\alpha}.
\end{eqnarray}

\textbf{Notation.}  The symbol $M \lesssim K$ means that there exists $\gamma> 0$ such that
$M \leq \gamma K$, where $\gamma$ is a constant.

\section{Direct problem for the q-heat equation}\label{S3}

We start to study  a Cauchy problem: 
\begin{eqnarray}\label{additive3.1}
D_{q,t}u(t)+\upsilon(t)\mathcal{L}u(t)=f(t), \;\;\;u\in{H},\;\;t>0,
\end{eqnarray} 
with the initial condition
\begin{eqnarray}\label{additive3.2}
 u(0)=\varphi \in{H}. 
\end{eqnarray}

\begin{thm}\label{thm3.1}
Assume that  Assumption \ref{a2.1} holds. Let $d\in\mathbb{R}$,  $\varphi \in \mathcal{H}^{d+2}_{\mathcal{L}}$ and $f\in L^\infty\left( [0, T]; \mathcal{H}^{d+2}_{\mathcal{L}}\right)$. Then there exists a unique solution of Problem (\ref{additive3.1})-(\ref{additive3.2}):
\begin{eqnarray}\label{additive3.3}
u\in C^1_q\left([0,T];\mathcal{H}^{d+2}_{\mathcal{L}}\right) \cap L^\infty\left([0,T];\mathcal{H}^{d+2}_{\mathcal{L}}\right).
\end{eqnarray}
Moreover, this solution can be written in the form
\begin{eqnarray}\label{additive3.4}
u(t)&=&\sum\limits_{k\in{I}} \left[ \langle\varphi,\phi_k\rangle_H \gamma_\upsilon(\lambda_kt)\right.\nonumber\\
&+&\left. \gamma_\upsilon (\lambda_kt)\int\limits_0^t \gamma^{-1}_\upsilon(q\lambda_ks)  \langle f(s),\phi_k\rangle_Hd_qs\right] \phi_k, 
\end{eqnarray}
which satisfies the estimate
\begin{eqnarray}\label{additive3.4A}
 \|u(t)\|_{\mathcal{H}^{d+2}_{\mathcal{L}}}+ \|D_qu(t)\|^2_{\mathcal{H}^d_{\mathcal{L}}} &\lesssim& 
 \| \varphi \|^2_{\mathcal{H}^{d+2}_{\mathcal{L}}}
+\|f(t)\|^2_{\mathcal{H}^{d+2}_{\mathcal{L}}}. 
\end{eqnarray}
\end{thm}

\begin{proof} {\it Existence.}  We start to solve the  equation  (\ref{additive3.1}). We can use the system of eigenfunctions, and look for a solution in the series form
\begin{eqnarray}\label{additive3.5}
 u(t)=\sum\limits_{k\in{I}} u_k(t)\phi_k,
\end{eqnarray} 
for each fixed $t > 0$. Such an expansion always exists due to completeness of the set of eigenfunctions  $\{\phi_k\}_{k=1}^\infty$ in $H$. The coefficients will then be given by the Fourier  coefficients formula: $ u_k(t)=\langle u(t), \phi_k\rangle_H$.

We can similarly expand the source function,
\begin{eqnarray}\label{additive3.6} 
f(t)=\sum\limits_{k\in{I}}f_k(t)\phi_k, \;\;\; 
 f_k(t)=\langle f(t), \phi_k\rangle_H.
\end{eqnarray}

From   (\ref{additive3.5}) we have $\mathcal{L}\phi_k=\lambda_k\phi_k, \; \; \; k\in{I}$. Hence,
\begin{eqnarray}\label{additive3.7} 
\mathcal{L}u(t)=\sum\limits_{k\in{I}} u_k(t) \lambda_k \phi_k
\end{eqnarray}
and 
\begin{eqnarray}\label{additive3.8} 
D_qu(t)=\sum\limits_{k\in{I}} D_qu_k(t)\phi_k.
\end{eqnarray}

Substituting (\ref{additive3.7}) and (\ref{additive3.8}) into the equation (\ref{additive3.1}),   we find 
\begin{eqnarray}\label{additive3.11} 
\sum\limits_{k\in{I}}  \left[D_qu_k(t)+\lambda_k\upsilon(t)u_k(t)\right]\phi_k=\sum\limits_{k\in{I}}  f_k(t)\phi_k.
\end{eqnarray}

But then, due to the completeness,
\begin{eqnarray}\label{additive3.12} 
D_qu_k(t)+\lambda_k\upsilon(t)u_k(t)= f_k(t), \;\;\;k\in{I},  
\end{eqnarray}  
 which are ODEs for the coefficients $u_k(t)$ of the series (\ref{additive3.5}). Using the integrating factor $\gamma^{-1}_\upsilon(\lambda_kt)$  and (\ref{additive2.2}) and (\ref{additive2.8}), we can rewrite this equation as
 
\begin{eqnarray}\label{additive3.13} 
\gamma^{-1}_\upsilon(q\lambda_kt)f_k(t)&=&\gamma^{-1}_\upsilon(q\lambda_kt)D_qu_k(t)+\gamma^{-1}_\upsilon(q\lambda_kt)\lambda_k\upsilon(t)u_k(t)\nonumber\\
&=&\gamma^{-1}_\upsilon(q\lambda_kt)D_qu_k(t)+D_q\left[\gamma^{-1}_\upsilon(\lambda_kt)\right] u_k(t)\nonumber\\
&=& D_q\left[\gamma^{-1}_\upsilon(\lambda_kt)u_k(t)\right].  
\end{eqnarray}

Form (\ref{additive2.4}) and (\ref{additive3.11}) we get
\begin{eqnarray*} 
\int\limits_0^t D_q\left[\gamma^{-1}_\upsilon(\lambda_ks)u_k(s)\right]d_qs&=& \int\limits_0^t\gamma^{-1}_\upsilon(q\lambda_ks)f_k(s)d_qs\\
&\Longrightarrow& \gamma^{-1}_\upsilon(\lambda_kt)u_k(t)=u_k(0)+ \int\limits_0^t\gamma^{-1}_\upsilon(q\lambda_ks)f_k(s)d_qs\\
&\Longrightarrow& u_k(t)=\gamma_\upsilon (\lambda_kt)u_k(0)+ \gamma_\upsilon (\lambda_kt)\int\limits_0^t\gamma^{-1}_\upsilon(q\lambda_ks)f_k(s)d_qs.
\end{eqnarray*}

But the initial condition (\ref{additive3.2}) and (\ref{additive3.5}) imply  
\begin{eqnarray*} 
 u(0)=\sum\limits_{k\in{I}} u_k(0)\phi_k=\varphi \;\;\;\Rightarrow\;\;\;  u_k(0)= \langle\varphi,\phi_k\rangle_H.
 \end{eqnarray*}
 
 Thus, 
\begin{eqnarray}\label{additive3.13} 
  u_k(t)= \langle\varphi,\phi_k\rangle_H\gamma_\upsilon(\lambda_kt)+  \gamma_\upsilon(\lambda_kt)\int\limits_0^t \gamma^{-1}_\upsilon(q\lambda_ks) \langle f(s),\phi_k\rangle_Hd_qs. 
\end{eqnarray}

So the solution can be written in the series form as
\begin{eqnarray*} 
u(t)=\sum\limits_{k\in{I}} \left[ \langle\varphi,\phi_k\rangle_H \gamma_\upsilon(\lambda_kt) + \gamma_\upsilon (\lambda_kt)\int\limits_0^t \gamma^{-1}_\upsilon(q\lambda_ks)  \langle f(s),\phi_k\rangle_Hd_qs\right] \phi_k, 
\end{eqnarray*}
 giving (\ref{additive3.4}).

{\it Convergence .} By using  Assumption \ref{a2.1} and (\ref{additive2.6A}) and (\ref{additive2.7A}) we obtain that 
\begin{eqnarray*}
\frac{\gamma_\upsilon(\lambda_kt)}{\gamma_\upsilon(\lambda_ks)}\overset{ \text{(2.7)} }{\leq}\frac{E_q^{\beta\lambda_ks}}{E_q^{\alpha\lambda_kt}}\leq
\frac{E_q^{\alpha\lambda_ks}}{E_q^{\alpha\lambda_kt}}\leq1, \;\;\;0<s\leq t,
\end{eqnarray*}
since $E_q^{\alpha\lambda_kt}$ is an increasing function for $t>0$. Indeed, using (\ref{additive2.6A}) we see that
\begin{eqnarray*}
\frac{d}{dt}\left[E_q^{\alpha\lambda_kt}\right] = 
\sum\limits_{m=1}^\infty q^{m(m-1)/2}\left[\alpha\lambda_k\right]^k\frac{ m}{[m]_q}t^{m-1}\geq 0.
\end{eqnarray*}
for $t>0$.

Hence, form (\ref{additive2.7A}) and  (\ref{additive3.13})     it  follows that 
\begin{eqnarray}\label{additive3.14} 
\left|\langle u(t),\phi_k\rangle_H\right|&\overset{ \text{(3.13)} }{\leq}&\gamma_\upsilon(\lambda_kt) \left|\langle\varphi,\phi_k\rangle_H\right| 
+  \int\limits_0^t\frac{\gamma_\upsilon(\lambda_kt)}{\gamma_\upsilon(\lambda_ks)}\left|\langle f(s),\phi_k\rangle_H\right|d_qs\nonumber\\
&\leq&\frac{1}{E_q^{\alpha{t}}}\left|\langle\varphi,\phi_k\rangle_H\right| 
+\int\limits_0^t\left|\langle f(s),\phi_k\rangle_H\right|d_qs\nonumber\\
&\leq&\max\left\{\frac{1}{E_q^{\alpha{t}}},1\right\}\left[\left|\langle\varphi,\phi_k\rangle_H\right| 
+\int\limits_0^t\left|\langle f(s),\phi_k\rangle_H\right|d_qs\right], 
\end{eqnarray}
and  using   (\ref{additive3.12}),  we get that
\begin{eqnarray}\label{additive3.15} 
\lambda^{d/2}_k\left|D_qu_k(t) \right|&\overset{ \text{(3.11)} }{\leq}& 
\lambda^{d/2+1}_k \upsilon(t) \left|\langle\varphi,\phi_k\rangle_H\right| +\lambda^{d/2}_k\left|\langle f(t),\phi_k\rangle_H\right|\nonumber\\
&\leq&\beta \left|\langle\lambda^{d/2+1}_k \varphi,\phi_k\rangle_H\right|+\lambda^{-1}_k \left|\langle \lambda^{d/2+1}_kf(t), \phi_k\rangle_H\right|\nonumber\\
&\leq&  \beta\left|\langle\mathcal{L}^{d/2+1}\varphi,\phi_k\rangle_H\right|  +\lambda^{-1}_0 \left|\langle \mathcal{L}^{d/2+1}f(s),\phi_k\rangle_H\right| \nonumber\\
&\leq&\max\{\beta,\lambda^{-1}_0\}\left[\left|\langle\mathcal{L}^{d/2+1}\varphi,\phi_k\rangle_H\right|+\sup\limits_{0\leq{s}\leq{T}}\left|\langle \mathcal{L}^{d/2+1}f(s),\phi_k\rangle_H\right|\right] 
\end{eqnarray} 
and  
\begin{eqnarray}\label{additive3.16} 
\lambda^{d/2}_k\left|\langle \mathcal{L}u(t),\phi_k\rangle_H\right| 
&=&\lambda^{d/2}_k\left|\langle \lambda_k u_k(t),\phi_k\rangle_H\right|  \nonumber\\
&\overset{ \text{(3.14)} }{\lesssim}&  \left|\langle\lambda^{d/2+1}_k \varphi,\phi_k\rangle_H\right| 
+\int\limits_0^t\left|\langle \lambda^{d/2+1}_k f(s),\phi_k\rangle_H\right|d_qs   \nonumber\\
 &=& \left|\langle\mathcal{L}^{d/2+1}\varphi,\phi_k\rangle_H\right| 
+\int\limits_0^t\left|\langle \mathcal{L}^{d/2+1}f(s),\phi_k\rangle_H\right|d_qs\nonumber\\
&\leq& \left|\langle\mathcal{L}^{d/2+1}\varphi,\phi_k\rangle_H\right| +T\sup\limits_{0\leq{s}\leq{T}}\left|\langle \mathcal{L}^{d/2+1}f(s),\phi_k\rangle_H\right| \nonumber\\
&\leq&\max\{1,T\}\left[\left|\langle\mathcal{L}^{d/2+1}\varphi,\phi_k\rangle_H\right|+\sup\limits_{0\leq{s}\leq{T}}\left|\langle \mathcal{L}^{d/2+1}f(s),\phi_k\rangle_H\right| \right].
\end{eqnarray}

Since $\varphi\in  \mathcal{H}^{d+2}_{\mathcal{L}}$, $f\in L^\infty\left([0, T]; \mathcal{H}^{d+2}_{\mathcal{L}}\right)$,  and using the Plancherel identity,  we have
 \begin{eqnarray*} 
\|D_qu(t)\|^2_{\mathcal{H}^d_{\mathcal{L}}}&=&\|\mathcal{L}^\frac{d}{2}D_qu(t)\|^2_H\\
&=&\sum\limits_{k\in{I}}\lambda_k^d\left|D_qu_k(t)\right|^2
\nonumber\\
&\overset{ \text{(3.15)} }{\lesssim}&  \sum\limits_{k\in{I}}\left|\langle\mathcal{L}^{d/2+1}\varphi,\phi_k\rangle_H\right|^2+ \sup\limits_{0\leq{s}\leq{T}}\sum\limits_{k\in{I}}\left|\langle \mathcal{L}^{d/2+1}f(s),\phi_k\rangle_H\right|^2 \nonumber\\
&=& \|\varphi\|^2_{\mathcal{H}^{d+2}_{\mathcal{L}}} +       \|f \|^2_{L^\infty\left( [0, T];\mathcal{H}^{d+2}_{\mathcal{L}}\right)} <\infty, 
\end{eqnarray*}
and  
\begin{eqnarray*} 
\|\mathcal{L}u(t)\|^2_{\mathcal{H}^d_{\mathcal{L}}}&=&\|\mathcal{L}^{d/2}\mathcal{L}u(t)\|^2_H\\&=&\sum\limits_{k\in{I}}\lambda^d_k\left|\langle \mathcal{L}u(t),\phi_k\rangle_H\right|^2\nonumber\\
&\overset{ \text{(3.16)} }{\lesssim}& \sum\limits_{k\in{I}}\left|\langle\mathcal{L}^{d/2+1}\varphi,\phi_k\rangle_H\right|^2+ \sup\limits_{0\leq{s}\leq{T}}\sum\limits_{k\in{I}}\left|\langle \mathcal{L}^{d/2+1}f(s),\phi_k\rangle_H\right|^2\nonumber\\
&\leq&  \| \varphi\|^2_{\mathcal{H}^{d+2}_{\mathcal{L}}} +\|f \|^2_{L^\infty\left( [0, T];\mathcal{H}^{d+2}_{\mathcal{L}}\right)}   <\infty.
\end{eqnarray*}

Hence, the above estimates imply that
\begin{eqnarray*} 
\|\mathcal{L}u(t)\|^2_{\mathcal{H}^d_{\mathcal{L}}}+\|D_qu(t)\|^2_{\mathcal{H}^d_{\mathcal{L}}}&\lesssim& \| \varphi\|^2_{\mathcal{H}^{d+2}_{\mathcal{L}}} +\|f \|^2_{L^\infty\left( [0, T];\mathcal{H}^{d+2}_{\mathcal{L}}\right)}  <\infty,
\end{eqnarray*}
which   means that $u\in C_q^1\left([0,T]; \mathcal{H}^{d+2}_{\mathcal{L}}\right) \cap L^\infty\left([0,T];\mathcal{H}^{d+2}_{\mathcal{L}}\right)$ and this yields the estimate (\ref{additive3.4A}).

{\it Uniqueness.} It only remains to prove the uniqueness of the solution. We assume the opposite, namely that there exist the functions $u(t)$ and $v(t)$, which are two different solutions of Problem (\ref{additive3.1})-(\ref{additive3.2}). Thus, we have that
$$
\left\{
  \begin{array}{ll}
    D_{q,t}u(t)+\varphi(t)\mathcal{L}u(t)=f(t),& \hbox{$t>0$}, \\
    u(0)=\varphi\in{H},  
  \end{array}
\right.
$$
and
$$
\left\{
  \begin{array}{ll}
    D_{q,t}v(t)+\varphi(t)\mathcal{L}v(t)=f(t),& \hbox{$  t>0$}, \\
    v(0)=\varphi\in{H}.
  \end{array}
\right.
$$

We define $W(t)=u(t)-v(t)$. Then the function $W(t)$ is a solution of the following problem
\begin{eqnarray*}
\left\{
  \begin{array}{ll}
    D_{q}w(t)+\varphi(t)\mathcal{L}w(t)=f(t),& \hbox{$t>0$}, \\
    w(0, x)=\varphi\in{H}.
  \end{array}
\right.
\end{eqnarray*}
From  (\ref{additive3.5}) it follows that $W(t)\equiv0$, that is,  $u(t)\equiv v(t)$ and this contradiction to our assumption proves the uniqueness of the solution. The proof is complete. 
\end{proof}

\section{ Inverse source problem}\label{S4}

In this subsection we consider the following problem: find a pair of functions $u(t)$ and $f$ in the space $H$ satisfying the  $q$-heat equation:
\begin{eqnarray}\label{additive4.1}
D_{q,t}u(t)+\upsilon(t)\mathcal{L}u(t)=g(t)f, \;\;\;f\in{H},\;\;t>0,
\end{eqnarray} 
with the initial condition 
\begin{eqnarray}\label{additive4.2}
 u(0)=\varphi\in{H}, 
\end{eqnarray} 
and the final condition;
\begin{eqnarray}\label{additive4.3}
u(T)=\eta\in{H}.
\end{eqnarray}

In the sequel we will make use of the following:
\begin{Assum}\label{a4.1} We assume that $g:[0, T]\rightarrow\mathbb{R}$ is a function  satisfying: 
\begin{itemize}
\item $g(s)>0$ for $0<s<T$.
\item $0<\alpha_0\leq\int\limits_0^Tg(s)d_qs\leq\beta_0<\infty$, 
\end{itemize}
where $\alpha_0,\beta_0>0$. 
\end{Assum}

\begin{thm}\label{thm4.1}
Assume that   Assumption \ref{a2.1} and  Assumption \ref{a4.1} hold.  Let $d\in\mathbb{R}$, $\varphi, \eta \in \mathcal{H}^{d+2}_{\mathcal{L}}$. Then there exists a unique solution of Problem (\ref{additive4.1})-(\ref{additive4.3}):
\begin{eqnarray*} 
u\in C^1_q\left([0,T];\mathcal{H}^d_{\mathcal{L}}\right) \cap L^\infty\left([0,T];\mathcal{H}^{d+2}_{\mathcal{L}}\right),\;\;\;
f\in   \mathcal{H}^{d+2}_{\mathcal{L}} .
\end{eqnarray*}
Moreover, this   $f$ and the solution $u$ can be represented by
\begin{eqnarray*} 
f &=&\sum\limits_{k\in{I}} \left[\frac{\eta_k\gamma^{-1}_\upsilon(\lambda_kT)-\langle\varphi,\phi_k\rangle_H}{\int\limits_0^T\gamma^{-1}_\upsilon(q\lambda_ks)g(s)d_qs}\right]\phi_k,
\end{eqnarray*}
and 
\begin{eqnarray*} 
u(t)&=&\sum\limits_{k\in{I}} \left\{\gamma_\upsilon(\lambda_kt)\left[\langle\varphi,\phi_k\rangle_H+ \langle f,\phi_k\rangle_H\int\limits_0^t\gamma^{-1}_\upsilon(q\lambda_ks)g(s)d_qs\right]\right\}\phi_k,  
\end{eqnarray*} 
respectively. Here the solution $u$ satisfies the estimate
\begin{eqnarray}\label{additive3.4A}
 \|u(t)\|_{\mathcal{H}^{d+2}_{\mathcal{L}}}+ \|D_qu(t)\|^2_{\mathcal{H}^d_{\mathcal{L}}} &\lesssim& 
 \| \varphi \|^2_{\mathcal{H}^{d+2}_{\mathcal{L}}}
+\|\eta\|^2_{\mathcal{H}^{d+2}_{\mathcal{L}}}. 
\end{eqnarray}
\end{thm} 

\begin{proof} {\it Existence.}  Since the system $\{\phi_k\}_{k\in{I}}$ is a basis in the space $H$, we expand the pair of functions $(u(t), f)$ as follows:
\begin{eqnarray}\label{additive4.4}
 u(t)=\sum\limits_{k\in{I}} u_k(t)\phi_k,\;\;\;
f=\sum\limits_{k\in{I}} f_k\phi_k,
\end{eqnarray} 
where $u_k(t)=\langle u(t), \phi_k\rangle_H$ and $u_k(t)=\langle f, \phi_k\rangle_H$.  
By repeating the arguments of Theorem \ref{thm3.1}, we start from the formula (\ref{additive3.12}). For the last term of the equation (\ref{additive3.12}), we have
\begin{eqnarray}\label{additive4.5} 
D_qu_k(t)+\lambda_k\upsilon(t)u_k(t)=g(t)f_k, \;\;\;k\in{I},  
\end{eqnarray}
and a general solution of the equation (\ref{additive4.6}) is given in the following form:
\begin{eqnarray}\label{additive4.6} 
u_k(t)=\gamma_\upsilon(\lambda_kt)\left[u_k(0)+ f_k\int\limits_0^t\gamma^{-1}_\upsilon(q\lambda_ks)g(s)d_qs\right], 
\end{eqnarray}
 where the constants $u_k(0), f_k$ are unknown. By using the conditions  (\ref{additive4.2}) and (\ref{additive4.3}) and Assumption \ref{a4.1}   we find  $ u_k(0)= \langle\varphi,\phi_k\rangle_H$ and
\begin{eqnarray}\label{additive4.7} 
\langle\eta,\phi_k\rangle_H=u_k(T)&=&\gamma_\upsilon (\lambda_kT)\left[\langle\varphi,\phi_k\rangle_H+  f_k\int\limits_0^T\gamma^{-1}_\upsilon(q\lambda_ks)g(s)d_qs\right].\nonumber\\
&\Rightarrow&f_k=\frac{\gamma^{-1}_\upsilon(\lambda_kT)\langle\eta,\phi_k\rangle_H-\langle\varphi,\phi_k\rangle_H}{\int\limits_0^T\gamma^{-1}_\upsilon(q\lambda_ks)g(s)d_qs}. 
\end{eqnarray}

Substituting $f_k$, $u_k(t)$ into the expansions (\ref{additive4.5}), we find that 
\begin{eqnarray*} 
f&=&\sum\limits_{k\in{I}} \left[\frac{\langle\eta,\phi_k\rangle_H\gamma^{-1}_\upsilon(\lambda_kT)-\langle\varphi,\phi_k\rangle_H}{\int\limits_0^T\gamma^{-1}_\upsilon(q\lambda_ks)g(s)d_qs}\right]\phi_k,
\end{eqnarray*}
and
\begin{eqnarray*} 
u(t)&=&\sum\limits_{k\in{I}} \left\{\gamma_\upsilon(\lambda_kt)\left[\langle\varphi,\phi_k\rangle_H+ f_k\int\limits_0^t\gamma^{-1}_\upsilon(q\lambda_ks)g(s)d_qs\right]\right\}\phi_k.
\end{eqnarray*}

{\it Convergence}. From Assumption \ref{a4.1} and  (\ref{additive2.7}) and  (\ref{additive2.7A})  we get that 
\begin{eqnarray}\label{additive4.8} 
\frac{\gamma^{-1}_\upsilon(\lambda_kT)}{\int\limits_0^T\gamma^{-1}_\upsilon(q\lambda_ks)g(s)d_qs}\leq\frac{E_q^{\beta{T}}}{\int\limits_0^T g(s)d_qs} 
\leq\frac{E_q^{\beta{T}}}{\alpha_0}.  
\end{eqnarray}

Using  (\ref{additive2.7A}), (\ref{additive4.7}) and  (\ref{additive4.8}) we find that 
 
\begin{eqnarray}\label{additive4.9}
\left|f_k\right|&\leq&\frac{\gamma^{-1}_\upsilon(\lambda_kT)}{\int\limits_0^T\gamma^{-1}_\upsilon(q\lambda_ks)g(s)d_qs}\left\{\left|\langle\eta,\phi_k\rangle_H\right|+\gamma _\upsilon(\lambda_kT)\left|\langle\varphi,\phi_k\rangle_H\right|\right\}\nonumber\\
&\overset{ \text{(4.9)} }{\leq}&\frac{E_q^{\beta{T}}}{\alpha_0}\left\{ \left|\langle\eta,\phi_k\rangle_H\right|+\frac{1}{E_q^{\alpha{T}}}\left|\langle\varphi,\phi_k\rangle_H\right|\right\} \nonumber\\
&\leq&\frac{E_q^{\beta{T}}}{\alpha_0}  \left\{ \left|\langle \eta,\phi_k\rangle_H\right|+\left|\langle\varphi,\phi_k\rangle_H\right| \right\}.
\end{eqnarray}

Hence, 
\begin{eqnarray*}
\|f\|^2_{\mathcal{H}^{d+2}_{\mathcal{L}}}&=&\sum\limits_{k\in{I}}\lambda_k^{d+2}\left|f_k\right|^2\nonumber\\
&\overset{ \text{(4.10)} }{\lesssim}&  \left|\langle \lambda^{d/2+1}_k\eta,\phi_k\rangle_H\right|^2+\left|\langle\lambda^{d/2+1}_k\varphi,\phi_k\rangle_H\right|^2\nonumber\\
&=&   \|\eta\|^2_{\mathcal{H}^{d+2}_{\mathcal{L}}}+ \| \varphi\|^2_{\mathcal{H}^{d+2}_{\mathcal{L}}}<\infty,
\end{eqnarray*} 
which   means that $f\in   \mathcal{H}^{d+2}_{\mathcal{L}}$.

Form (\ref{additive2.7A}), (\ref{additive4.6}) and (\ref{additive4.9})  and Assumption \ref{a4.1}  it  follows that 
\begin{eqnarray}\label{additive4.10}
|u_k(t)|&\leq& \frac{\left|\langle\varphi,\phi_k\rangle_H\right|+  \left|f_k\right|\int\limits_0^t\gamma^{-1}_\upsilon(q\lambda_ks)g(s)d_qs}{E_q^{\alpha{t}}} \nonumber\\
&\leq& \frac{\left|\langle\varphi,\phi_k\rangle_H\right|+ E_q^{\beta{t}}\beta_0\left|f_k\right| }{E_q^{\alpha{t}}}\nonumber\\
&\leq&\frac{1}{E_q^{\alpha{t}}}\left|\langle\varphi,\phi_k\rangle_H\right|
+\beta_0\frac{E_q^{\beta{T}}}{E_q^{\alpha{t}}}\left|f_k\right|\nonumber\\
&\overset{ \text{(4.10)} }{\lesssim}& \left|\langle\varphi,\phi_k\rangle_H\right|+|\langle\eta,\phi_k\rangle_H|.  
\end{eqnarray}

By Assumption \ref{a2.1} and  (\ref{additive4.5})   and (\ref{additive4.10})  we get that 
\begin{eqnarray}\label{additive4.11}
\lambda^{d/2}_k\left|D_qu_k(t)\right|&\leq&\lambda^{d/2+1}_k|\upsilon(t)||u_k(t)|+\lambda^{d/2}_k|g(t)||f_k|\nonumber\\
&\leq&\beta\lambda^{d/2+1}_k|u_k(t)|+\frac{\beta_0}{\lambda_0}\lambda^{d/2+1}_k|f_k|\nonumber\\
&\overset{ \text{(4.10), (4.11)} }{\lesssim}&  \left|\langle\lambda^{d/2+1}_k\eta,\phi_k\rangle_H\right|+ \left|\langle\lambda^{d/2+1}_k\varphi,\phi_k\rangle_H\right| \nonumber\\
&=&  \left|\langle\mathcal{L}^{d/2+1}\eta,\phi_k\rangle_H\right|+ \left|\langle\mathcal{L}^{d/2+1}\varphi,\phi_k\rangle_H\right|.
\end{eqnarray}

Since $\eta,\varphi\in  \mathcal{H}^{d+2}_{\mathcal{L}}$   and the Plancherel identity  we have that 
\begin{eqnarray*} 
\|D_qu(t)\|^2_{\mathcal{H}^{d}_{\mathcal{L}}}&=&\sum\limits_{k\in{I}}\lambda_k^d\left|D_qu_k(t)\right|^2
\nonumber\\
&\overset{ \text{(4.12)} }{\lesssim}&\sum\limits_{k\in{I}}\left[\left|\langle\mathcal{L}^{d/2+1}\eta,\phi_k\rangle_H\right|^2+ \left|\langle\mathcal{L}^{d/2+1}\varphi,\phi_k\rangle_H\right|^2\right] \nonumber\\
&=&  \|\eta\|^2_{\mathcal{H}^{d+2}_{\mathcal{L}}}+\| \varphi\|^2_{\mathcal{H}^{d+2}_{\mathcal{L}}}<\infty, 
\end{eqnarray*}
and  
\begin{eqnarray*} 
\|\mathcal{L}u(t)\|^2_{\mathcal{H}^{d}_{\mathcal{L}}}&=&\sum\limits_{k\in{I}}\lambda_k^d\left|\langle \mathcal{L}u(t),\phi_k\rangle_H\right|^2\nonumber\\
&=&\sum\limits_{k\in{I}}\left[\lambda^{d/2+1}_k\left|u_k(t)\right|\right]^2\nonumber\\
&\lesssim& \sum\limits_{k\in{I}}\left[  \left|\langle\mathcal{L}^{d/2+1}\eta,\phi_k\rangle_H\right|^2+ \left|\langle\mathcal{L}^{d/2+1}\varphi,\phi_k\rangle_H\right|^2 \right]\nonumber\\
&=&   \|\eta\|^2_{\mathcal{H}^{d+2}_{\mathcal{L}}}+ \| \varphi\|^2_{\mathcal{H}^{d+2}_{\mathcal{L}}}   <\infty.
\end{eqnarray*}

Hence, the above estimates imply that
  $u\in C_q^1\left([0,T];\mathcal{H}^d_{\mathcal{L}}\right) \cap L^\infty\left([0,T];\mathcal{H}^{d+2}_{\mathcal{L}}\right)$ and  this yields the estimate (\ref{additive3.4A}).

{\it Uniqueness.} The part is similar above theorem \ref{thm3.1}. The proof is completed. 
\end{proof}

\section{Examples}\label{S5}

In this section we give several examples of the settings where our direct and inverse problems are applicable.

{$\blacksquare$ \it The q-Sturm–Liouville problem}: Let $H=L^2_q\left[0, a\right]$ be the space of all real-valued functions defined on $[0, a]$ such that 
\begin{eqnarray*}
\|f\|_{L^2_q\left[0, a\right]}:= \left(\int\limits_0^a |f(x)|^2d_qx\right)^\frac{1}{2}<\infty.  
\end{eqnarray*}

The space  $L^2_q\left[0, a\right]$ is a separable Hilbert space with the inner product:
\begin{eqnarray*}
\langle f,g\rangle:= \int\limits_0^a f(x) g(x)d_qx, \;\;\;f,g\in  L^2_q\left[0, a\right].
\end{eqnarray*}

Moreover, we   denote $C^2_{q,0}[0,a]$ the  space   of
all functions $y(\cdot)$  such that  $y$, $D_qy$ are continuous at zero.

M.~H.~Annaby and Z.~S.~Mansour considered  a basic $q$-Sturm–Liouville eigenvalue problem in the  Hilbert space $L^2_q\left[0, a\right]$ \cite[Chapter 3]{AM2005}:
\begin{equation}\label{additive5.1}
\mathcal{L}(y)=\left\{ \begin{array}{cl}
-\frac{1}{q}D_{q^{-1}}D_qy(x)+v(x)y(x)=\lambda y(x),\\
U_1(y)=a_{11}y(0)+a_{12}y(0),\\
U_2(y)=a_{21}y(a)+a_{22}y(a),
\end{array}\right.
\end{equation}
for $0\leq{x}\leq{a}<\infty$ and $\lambda\in\mathbb{C}$, where $v(\cdot)$  is a continuous at zero real valued function and $a_{ij}$,  $i,j\in\{1,2\}$ are
arbitrary real numbers such that the rank of the matrix $(a_{ij})_{1\leq{i,j}\leq2}$ is 2. The basic Sturm- Liouville eigenvalue problem (\ref{additive5.1})   is self adjoint on  $C^2_{q,0}[0,a]\cap L^2_q\left[0, 1\right]$(see \cite[Theorem 3.4.]{AM2012}). The eigenvalues and the eigenfunctions of the   problem (\ref{additive5.1}) have the following properties:
\begin{itemize}
\item The eigenvalues are real.
\item Eigenfunctions that correspond to different eigenvalues are orthogonal.
\item All eigenvalues are simple.  
\end{itemize}

{$\blacksquare$ \it The q-Bessel operator}. Let  $\alpha>-1$ and $1\leq p < \infty$. Then the space $L_{\alpha,p,q}$ denotes the set of functions on $\mathbb{R}^+_q=\{q^z: z\in\mathbb{Z}\}$
such that
\begin{eqnarray*}
\|f\|_{L_{\alpha,p,q}}=\left(\int\limits_0^\infty\left|f(x)\right|^px^{2\alpha+1}d_qx\right)^\frac{1}{p}<\infty. 
\end{eqnarray*}

The set $H=L_{\alpha,p,q}$ is an Hilbert space with the inner product
\begin{eqnarray*} 
\langle f,g\rangle_{L_{\alpha,p,q}}= \int\limits_0^\infty  f(x)g(x)x^{2\alpha+1}d_qx.
\end{eqnarray*}

Moreover, we introduce the space $C^k_{q,0}$ for $k\in\mathbb{N}$:
\begin{eqnarray*} 
C^k_{q,0}\left(\mathbb{R}^+_q\right)=\{f:\mathbb{R}^+_q\rightarrow\mathbb{R};    \sup\limits_{x\in\mathbb{R}^+_q}\left|D^k_qf(x)\right|<\infty \; \text{and} \;  D^k_qf(0)=D^k_qf(\infty) =0\}.
\end{eqnarray*}  
The $q$-Bessel operator is defined as follows (see \cite{LAJ2006} and \cite{HR1994}):
\begin{eqnarray}\label{additive5.2}
\Delta_{q,\alpha}f(x)&=&\frac{1}{x^{2\alpha+1}}D_q\left[x^{2\alpha+1}D_qf\right](q^{-1}x)\nonumber\\
&=&q^{2\alpha+1}\Delta_qf(x)+\frac{1-q^{2\alpha+1}}
{(1-q)q^{-1}x}D_qf(q^{-1}x),
\end{eqnarray}
where
\begin{eqnarray*}
\Delta_qf(x)=D_q^2f(q^{-1}x). 
\end{eqnarray*}

For $f,g\in C^2_{q,0}\left(\mathbb{R}^+_q\right)$, using formulas (\ref{additive5.2}) and  (\ref{additive2.2}) of $q$-integration by parts,   we obtain
\begin{eqnarray*} 
\langle \Delta_{q,\alpha}f,g\rangle_{L_{q,2,\alpha}}&=& \int\limits_0^\infty  D_q\left[x^{2\alpha+1}D_qf\right](q^{-1}x)g(x) d_qx \\
&=&\left[x^{2\alpha+1}D_qf(x)g(x)\right]_0^\infty-\int\limits_0^\infty   \left[x^{2\alpha+1}D_qf\right](x)D_qg(x)d_qx\\
&=&-\int\limits_0^\infty   D_qf(x)\left[x^{2\alpha+1}D_qg\right](x)d_qx\\
&=&-\left[D_qf(x)\left[x^{2\alpha+1}D_qg\right](x)\right]_0^\infty+\int\limits_0^\infty  f(qx)D_q\left[x^{2\alpha+1}D_qg\right](x)d_qx\\
&=&\int\limits_0^\infty  f(x)D_q\left[x^{2\alpha+1}D_qg\right](q^{-1}x)d_qx\\
&=&\langle f,\Delta_{q,\alpha}g\rangle_{L_{q,2,\alpha}}. 
\end{eqnarray*}

It follows from \cite[Proposition 1]{FHB2002},  that the function 
\begin{eqnarray*}
j_\alpha(\lambda x; q^2)=\Gamma_{q^2}(\alpha+1)\sum\limits_{k=0}^\infty\frac{(-1)^kq^{k(k-1)}}{\Gamma_{q^2}(\alpha+k+1)\Gamma_{q^2}(k+1)}\left(\frac{x}{1+q}\right)^k, 
\end{eqnarray*}
with the eigenvalue $\lambda\in\mathbb{C}$  for the eigenfunction $x\mapsto j_\alpha(\lambda x; q^2)$, where $j_\alpha(\cdot, q^2)$ is called the normalized $q$-Bessel function defined by \cite{LAJ2006} and \cite{HR1994}.

{$\blacksquare$ \it The $q$-deformed
Hamiltonian}. The field of $q$-deformed oscillator algebras and quantum orthogonal polynomials continues to be at the core of intense activities in physics and mathematics. In 2015, W.~S.~Chung and M.~N.~ Hounkonnou and A.~Sama (\cite{WMA2015}) constructed the $q$-deformed Hamiltonian. For the clarity of our exposition, let us briefly discuss in this section main relevant results on $q$-Hermite functions. Let $\mathcal{H}_F$ be the Hilbert space spanned by the basis vectors $\{\psi^q_n(x), n = 1, 2,\cdots\}$ such that
\begin{eqnarray*}
\psi^q_n(x)=\frac{H_n^q(x)}{\sqrt{[2]^2_q[n]_q!}}=\frac{1}{\sqrt{[2]_q^n[n]_q!}}\sum\limits_{k=0}^{\langle\frac{n}{2}\rangle}\times\frac{(-1)^kq^{k(k-1)[n]_q!}}{[n-2k]_q![k]_{q^2}!}\left([2]_qx\right)^{n-2k}, \;\;\;n=0,1,2,\cdots,
\end{eqnarray*}
where the $q$-analogue of the binomial coefficients $[n]_{q}!$ are defined by
\begin{equation*}
[n]_{q}!:=\left\{
\begin{array}{l}
{1,\mathrm{\;\;\;\;\;\;\;\;\;\;\;\;\;\;\;\;\;\;\;\;\;\;\;\;\;\;\;\;\;\;\;\;%
\;if\;{\it n}}=\mathrm{0,}} \\
{[1]_{q}\times [2]_{q}\times \cdots \times [n]_{q},\mathrm{%
\;if\;{\it n}}\in \mathrm{ \mathbb{N},\;\;}}%
\end{array} \right.
\end{equation*}

Moreover, we can write 
\begin{eqnarray*}
D_q\psi^q_n(x)=\sqrt{[2]_q[n]_q}\psi^q_{n-1}(x). 
\end{eqnarray*}

The vectors $|n\rangle$ are eigen-vectors of the $q$-deformed Hamiltonian (see \cite[Proposition 1]{WMA2015})
\begin{eqnarray*}
H_q=\frac{1}{[2]_q}\left(AA^\dagger+A^\dagger A\right)
\end{eqnarray*}
with respect to the eigenvalues
\begin{eqnarray*}
E_n^q=\frac{1}{[2]_q}\left([n]_q+[n+1]_q\right),
\end{eqnarray*}
where the annihilation (lowering) $A$ and
creation (raising) $A^\dagger$ operators are 
\begin{eqnarray*}
A=\frac{1}{\sqrt{[2]_q}}D_q, \;\;\; A^\dagger=\sqrt{[2]_q}x-\frac{q^{N-1}}{[2]_q}D_q.
\end{eqnarray*}

In the limit when $q\rightarrow1$, one recovers the uncertainty
relation for the non deformed harmonic oscillator (or the classical case).

{$\blacksquare$ \it Fractional Sturm-Liouville operator}.

Let $1/2<\alpha\leq1$ and $[a, b]\subset\mathbb{R}$. Then we consider   a fractional problem with boundary conditions in the form (see \cite{RTV2013}): 
\begin{equation}\label{additive5.3}
\mathcal{L}_a\alpha(y)=\left\{ \begin{array}{cl}
-D^\alpha_{b-}\left(\rho(x)D^\alpha_{a+}y(x)\right)(x)+\mu(x)y(x)=\lambda \tau(x)y(x),\\
 a_{11}y(a)+a_{12}I_{b-}^{1-\alpha}\left(\rho D_{a+}^\alpha y\right)|_{x=a}=0,\\
a_{21}y(b)+a_{22}I_{b-}^{1-\alpha}\left(\rho D_{a+}^\alpha y\right)|_{x=b}=0,
\end{array}\right.
\end{equation}
where $D^\alpha_{b-}$ is the right-sided and $D^\alpha_{a+}$ is left-sided Riemann-Liouville fractional derivatives, $I_{b-}^{1-\alpha}$ is the right-sided Riemann-Liouville fractional integral, (\ref{additive5.3}) is a self-adjoint operator in $H=L^2[a,b]$, the
constants in the boundary conditions verify $a_{11}^2+a_{12}^2 \neq0$,  $a_{21}^2+a_{22}^2 \neq0$ and $\rho$, $\mu$ and $\tau$ are continuous functions, such
that $\rho(x) > 0$ and $\tau(x) > 0$ for all  $a\leq{x}\leq{b}$. The function $\tau$ is called the “weight” or “density” function and the
values of $\lambda$  are called eigenvalues of the frational  boundary value problem.

{$\blacksquare$ \it The restricted fractional Laplacian}.

In \cite{CS2017}, L.~A.~Caffarelli and Y.~Sire introduced a fractional Laplacian operator by using the
integral representation in terms of hypersingular kernels in the following form: 
\begin{eqnarray}\label{additive5.4}
\left(-\Delta_{\mathbb{R}^n}\right)^\alpha=C_{d,\alpha}P.V.\int\limits_{\mathbb{R}^n}\frac{f(x)-f(t)}{\left|x-t\right|^{n+2\alpha}}dt,\;\;\;0<\alpha<1.
\end{eqnarray}

The operator (\ref{additive5.4}) is a self-adjoint in $L_2\left(\Omega\right)$ with  eigenvalues $\lambda_{\alpha,k}>0$, $k\in\mathbb{N}$. The corresponding set of eigenfunctions $\psi_{\alpha,k}$, $k\in\mathbb{N}$,  in $L_2\left(\Omega\right)$ (bounded domain $\Omega \subset\mathbb{R}^n$).

{$\blacksquare$ \it Differential operator with involution}.

As a next example, we consider the differential operator with involution in
$L^2(0, \pi)$ generated by the expression
\begin{eqnarray}\label{additive5.5}
\ell(u)=u''(x)-\epsilon u''(\pi-x)
\end{eqnarray}
for $0<x<\pi$,  with homogeneous Dirichlet  conditions 
\begin{eqnarray}\label{additive5.6}
u(0)=0, \;\;u(\pi)=0,
\end{eqnarray}
where $|\epsilon|<1$. The nonlocal functional-differential operator (\ref{additive5.5})-(\ref{additive5.6})  is self-adjoint (see, \cite{KST2017}). For $|\epsilon|<1$, the operator (\ref{additive5.5})-(\ref{additive5.6}) has the following eigenvalues:
\begin{eqnarray*}
\lambda_{2k}=4(1+\epsilon)k^2, k\in\mathbb{N}\;\text{and }\;\lambda_{2k+1}=(1-\epsilon)(2k+1)^2, k\in\mathbb{N}\cup\{0\},
\end{eqnarray*}
and corresponding eigenfunctions
\begin{eqnarray*}
u_{2k}&=&\sqrt{\frac{2}{\pi}}\sin2kx,\;\; k\in\mathbb{N}\\
u_{2k+1}&=&\sqrt{\frac{2}{\pi}}\sin(2k+1)x,\;\; k\in\mathbb{N}\cup\{0\}. 
\end{eqnarray*}

{$\blacksquare$ \it Landau Hamiltonian in 2D}.

The next example is one of the simplest and most interesting models of the
quantum mechanics, that is, the Landau Hamiltonian. The Landau Hamiltonian in 2D is given by
\begin{eqnarray}\label{additive5.5}
\mathcal{L}=\frac{1}{2}\left(\left(i\frac{\partial}{\partial{x}}-By\right)^2+\left(i\frac{\partial}{\partial{y}}+By\right)^2\right),
\end{eqnarray}
acting on the Hilbert space $L^2\left(\mathbb{R}\right)$, where $B > 0$ is some constant. The
spectrum of $\mathcal{L}$ consists of infinite number of eigenvalues  with infinite multiplicity of the form (see, \cite{V1928} and \cite{L1930}): 
\begin{eqnarray*}
\lambda_k=(2n+1)B, n=0,1,2,\cdots
\end{eqnarray*}
and the corresponding system of eigenfunctions (see, \cite{LPZ2015} and \cite{AH2013}) is
\begin{eqnarray*}
 \left\{ \begin{array}{rcl}
 e)1_{k,n}(x,y)=\sqrt{\frac{n!}{(n-k)!}}B^{\frac{k+1}{2}}\exp\left(-\frac{B(x^+y^2)}{2}\right)(x+iy)^kL_n^{(k)}\left(B(x^2+y^2)\right) & \mbox{for} & 0\leq k \\
e)1_{k,n}(x,y)=\sqrt{\frac{n!}{(n-k)!}}B^{\frac{k+1}{2}}\exp\left(-\frac{B(x^+y^2)}{2}\right)(x+iy)^kL_n^{(k)}\left(B(x^2+y^2)\right) & \mbox{for} & 0\leq j,
\end{array}\right.
\end{eqnarray*}
where $L_n^{(\alpha)}$
 are the Laguerre polynomials given by
\begin{eqnarray*}
L_n^{(\alpha)}=\sum\limits_{k=0}^\infty(-1)^kC_{n+\alpha}^{m-k}\frac{t^k}{k!},\;\;\alpha>-1.
\end{eqnarray*}

Note that in \cite{RT17b}, \cite{RT18} and  \cite{RT19a} the wave equation for the Landau Hamiltonian with a singular magnetic field was studied.

 \end{document}